\documentclass{amsart}
\usepackage[english]{babel}
\usepackage{amsfonts, amsmath, amsthm, amssymb,amscd,indentfirst}

\usepackage{esint}
\usepackage{graphicx}
\usepackage{epstopdf}
\usepackage{amsmath,amssymb,latexsym,indentfirst}
\usepackage{times}
\usepackage{palatino}
\usepackage[colorlinks]{hyperref}

\newtheorem{theorem}{Theorem}[section]

\newtheorem{definition}[theorem]{Definition}
\newtheorem{corollary}[theorem]{Corollary}
\newtheorem{remark}[theorem]{Remark}

\def\d{\partial}

\def\bp{\begin{proof}}
\def\ep{\end{proof}}

\def\d{\partial}

\begin{document}

\title[Rigidity on an eigenvalue problem with mixed boundary condition]{Rigidity on an eigenvalue problem with mixed boundary condition}

\author{\textsc{S\'{e}rgio Almaraz}}
\address{Universidade Federal Fluminense (UFF), Instituto de Matem\'{a}tica, Campus do Gragoat\'a\\
              Rua Prof. Marcos Waldemar de Freitas S/N ,  24210-201, Niter\'{o}i-RJ, Brazil.}
              \email{sergio.m.almaraz@gmail.com}
\author{\textsc{Ezequiel Barbosa}}
\address{Universidade Federal de Minas Gerais (UFMG), Departamento de Matem\'{a}tica, Caixa Postal 702, 30123-970, Belo Horizonte-MG, Brazil.}
\email{ezequiel@mat.ufmg.br}
\thanks{The two authors were partially supported by  CNPq/Brazil grants.}

\begin{abstract}
We prove an Obata-type rigidity result for the spherical cap and apply it for an eigenvalue problem with mixed boundary condition.

MSC classes: 53C24, 35P15.
\end{abstract}

\maketitle

\section{Introduction and statements of the results}\label{int}

Let $(M^n,g)$ be a compact Riemannian manifold, with a possibly non-empty boundary $\Sigma$.  We denote by $\text{Ric}_g$ and $R_g$ respectively the Ricci and scalar curvatures of $(M, g)$. When $\Sigma\neq \emptyset$, we also assume that $\Sigma$ is oriented by an outward pointing unit normal vector $\nu$, so that  its mean curvature is $H_{\Sigma}=\text{div}_g\nu$.

The classical Obata's rigidity theorem is a very well-known result with many important applications:

\begin{theorem}[\cite{O}]\label{thm:obata}
A connected closed Riemannian manifold $(M^n,g)$ is isometric to the standard unit sphere $(\mathbb{S}^n,g_0)$ if and only if there is a non-constant function $\phi$ satisfying
\[
\nabla^2 \phi+\phi g=0,
\]
where $\nabla^2$ is the Hessian with respect to the metric $g$.
\end{theorem}
\begin{corollary}
If a connected closed Riemannian manifold $(M^n,g)$ satisfies 
$$\text{Ric}_g\geq (n-1)g,$$ 
then the first eigenvalue $\lambda_1(g)$ of the Laplacian $\Delta_g$ satisfies
$\lambda_1(g)\geq n$, and equality holds if and only if $(M,g)$ is isometric to $(\mathbb{S}^n,g_0)$.
\end{corollary}

To the best of our knowledge, there are two versions of those results for manifolds with non-empty boundary: 

\begin{theorem}[\cite{Re1, Re2}]
A connected compact Riemannian manifold $(M^n,g)$ with non-empty boundary $\Sigma$  is isometric to the standard unit hemisphere $(\mathbb{S}^n_+,g_0)$ if and only if there is a non-constant function $\phi$ satisfying
\[
\begin{cases}
\nabla^2 \phi+\phi g=0\:\:\:\text{in}\:M,
\\
\phi= \text{const.\:\:\:on\:} \Sigma.
\end{cases}
\]
\end{theorem}
\begin{corollary}
Let $(M^n,g)$ be a connected compact Riemannian manifold with nonempty boundary $\Sigma$. Assume that $\text{Ric}_g\geq(n-1)g$ and $H_{\Sigma}\geq 0$. Then the first eigenvalue $\lambda_1(g)$ of $\Delta_g$ with Dirichlet boundary condition satisfies $\lambda_1(g)\geq n$. Moreover, $\lambda_1(g)=n$ if and only if $(M,g)$ is isometric to $(\mathbb{S}^n_+,g_0)$.
\end{corollary}
\begin{theorem}[\cite{Escobar, Xia}]
A connected compact Riemannian manifold $(M^n,g)$ with non-empty boundary  $\Sigma$ is isometric to  $(\mathbb{S}^n_+,g_0)$ if and only if \[
\begin{cases}
\nabla^2 \phi+\phi g=0\:\:\:\text{in}\:M,
\\
\frac{\partial}{\partial \nu}\phi=0\text{\:\:\:on\:} \Sigma.
\end{cases}
\]
\end{theorem}
\begin{corollary}
Let $(M^n,g)$ be a connected compact Riemannian manifold with nonempty boundary $\Sigma$. Assume that $\text{Ric}_g\geq(n-1)g$ and $H_{\Sigma}\geq 0$. Then the first eigenvalue $\lambda_1(g)$ of $\Delta_g$ with Neumann boundary condition  satisfies $\lambda_1(g)\geq n$. Moreover $\lambda_1(g)=n$ if and only if $(M,g)$ is isometric to $(\mathbb{S}^n_+,g_0)$.
\end{corollary}

Our main result extends the above to a mixed boundary condition, under certain hypothesis on the boundary mean curvature:

\begin{theorem}\label{main:thm}
Let $\theta\in (0,\frac{\pi}{2})$. A connected compact Riemannian manifold $(M^n,g)$ with non-empty $\Sigma$ satifying $H_{\Sigma}\geq (n-1)\tan(\theta)$ is isometric to a geodesic ball in $(\mathbb{S}^n,g_0)$, whose boundary is a totally umbilical hypersphere with mean curvature $(n-1)\tan(\theta)$, if and only if there is a non-constant function $\phi$ satisfying
\[
\begin{cases}
\nabla^2 \phi+\phi g=0\:\:\:\text{in}\:M,
\\
\tan(\theta)\frac{\partial}{\partial \nu}\phi+\phi=0\text{\:\:\:on\:} \Sigma.
\end{cases}
\]
\end{theorem}
As an application, we are able to complete the study initiated by Xin'an and Hongwei \cite{XinHongwei} on the first eigenvalue of the Laplacian operator with mixed boundary condition. Combining the inequality proved in \cite{XinHongwei} with our rigidity statement, we obtain

\begin{theorem}\label{main:application}
Let $(M^n,g)$ be a connected compact Riemannian manifold with nonempty boundary $\Sigma$ whose 2nd fundamental form and mean curvature are denoted by $h_{\Sigma}$ and $H_{\Sigma}=\text{tr}\,h_{\Sigma}$ respectively. If for some $\theta\in (0,\frac{\pi}{2})$,
$$\text{Ric}_g\geq (n-1)g,  \:\:\:h_{\Sigma}\geq-2\cot(\theta)g \:\:\text{and}\:\: H_{\Sigma}\geq(n-1)\tan(\theta),$$ 
then the first eigenvalue $\lambda_1(g,\theta)$ of the problem
\begin{equation}\label{eigenvalue:problem}
\begin{cases}
\Delta_g\phi+\lambda\phi=0&\text{in}\:M,
\\
\tan(\theta)\frac{\partial}{\partial\nu}\phi+\phi=0&\text{on}\:\Sigma,
\end{cases}
\end{equation}
satisfies $\lambda_1(g,\theta)\geq n$. Moreover, the equality $\lambda_1(g,\theta)=n$ holds if and only if $M$ is isometric to an $n$-dimensional spherical cap of $(\mathbb{S}^n, g_0)$, whose boundary is a totally umbilical hypersphere with mean curvature $(n-1)\tan(\theta)$.
\end{theorem}

Our contribution is the equality part, since the inequality was proved in \cite{XinHongwei}. 

In Section 2 we prove Theorem \ref{main:thm} and in Section 3 we prove the rigidity part of Theorem \ref{main:application}


\section{Proof of Theorem \ref{main:thm}}
Without loss of generality, we assume that $\phi(p)=\max\limits_{M}\phi=1$ for some $p\in M$. Since $\tan(\theta)>0$, it follows from  the second equation in \eqref{eigenvalue:problem} that $p\notin\Sigma$. 

Let $\gamma:[0,l_0]\rightarrow M$ be a geodesic such that $|\gamma'|=1$, $\gamma(0)=p$. Define $u(t)=\phi\circ \gamma(t)$. Since $\nabla^2 \phi+\phi g=0$, we obtain the o.d.e.
\begin{equation*}
\begin{cases}
u''+u=0,
\\
u(0)=1, \:\:u'(0)=0.
\end{cases}
\end{equation*}
Hence, $u(t)=\cos t$ in $[0,l_0]$.
Since $M$ is compact, we can choose $\gamma$ in such a way that $l_0=dist_g(p,\Sigma)=dist_g(p,q)$, where $q=\gamma(l_0)\in \Sigma$. 

Set
\[
\overline{l}=\sup \{l\in (0,l_0]\,|\,\exp_{p}: B_l(p)\subset M\rightarrow B_l(0)\subset \mathbb{R}^n \,\,\mbox{is a diffeomorphism} \}\,.
\]
Observe that $\overline{l}>0$, because the exponential map $\exp_{p}: B_l(p)\subset M\rightarrow B_l(0)\subset \mathbb{R}^n$ is a diffeomorphism for small $l$. 
Now, taking any point $p_0\in \mathbb{S}^n$ and using the exponential map on $S^n$ at $p_0$, we obtain a diffeomorphism  
$$
\psi: B_l(p)\subset M\rightarrow B_l(p_0)\subset \mathbb{S}^n$$ 
for any $0<l<\overline{l}$ with $l<\pi$. As in the proof of the classical Obata's Theorem \cite[pp.338-339]{O}, we can prove that $\psi$ is an isometry. So $\overline{l}<\pi$, otherwise $M$ is closed which is a contradiction.

It follows from definition that $\overline{l}\leq l_0$. We claim that $\overline{l}=l_0$. Suppose by contradiction $\overline{l}< l_0$. Then there exists a geodesic starting at $p$ with a cut-locus point in $\bar  B_{\overline{l}}(p)$. Hence, we see that $\overline{l}\geq\pi$ which is a contradiction.
Thus, $\overline{l}=l_0$.

\vspace{0.1cm}
\emph{Claim:} We have $H_{\partial B_{l_0}(p)}=\tan(\theta)(n-1)$.

Since $\psi$ is an isometry, it is enough to prove that $\overline l=\pi/2-\theta$. Note that $u(t)=\cos(t)$ implies
\begin{equation}\label{eq:1}
-\tan(\theta)\frac{\partial}{\partial \nu}\phi(q)=-\tan(\theta)u'(\overline{l})=\tan(\theta)\sin (\overline{l})
\end{equation}
because $\gamma$ intersects $\Sigma$ orthogonally at $\gamma(\overline{l})$.
On the other hand, it holds
\begin{equation}\label{eq:2}
\phi(q)=u(\overline{l})=\cos(\overline{l})
\end{equation}
It follows from the second equation in \eqref{eigenvalue:problem} that the l.h.s. of \eqref{eq:1} and \eqref{eq:2} equal, so 
$$
\tan(\theta)=\cot(\overline{l})=\tan(\pi/2-\overline{l}).
$$
Now, because $\tan$ is an injective function in $(-\pi/2,\pi/2)$, we have $\overline{l}=\pi/2-\theta$.
This proves the claim.

We are now able to finish the proof of Theorem \ref{main:thm}. Using $\overline{l}=l_0=dist_g(p,q)$ we obtain $\Sigma \subset M\setminus B_{l_0}(p)$ and $q\in \partial B_{l_0}(p)\cap \Sigma$. Using $H_{\partial B_{l_0}(p)}=\tan(\theta)(n-1)\leq H_{\Sigma}$ and the maximum principle, we conclude that $\partial B_{l_0}(p)=\Sigma$ and, consequently, $M=\bar B_{l_0}(p)$ which is isometric to $\bar B_{l_0}(p_0)\subset \mathbb{S}^n$.



\section{The rigidity part in Xin'an-Hongwei's theorem}

\begin{theorem}[\cite{XinHongwei}]\label{thm:XinHongwei}
\footnote{
Although our hypothesis on $h_{\Sigma}$ are weaker than stated in \cite{XinHongwei} (which assumes $h_{\Sigma}\geq 0$), the proof of Theorem \ref{thm:XinHongwei} follows exactly the same lines as \cite{XinHongwei}.
}
Let $(M^n,g)$ be a connected compact Riemannian manifold with nonempty boundary $\Sigma$. If for some $\theta\in (0,\frac{\pi}{2})$
$$\text{Ric}_g\geq (n-1)g,  \:\:\:h_{\Sigma}\geq-2\cot(\theta)g \:\:\text{and}\:\: H_{\Sigma}\geq(n-1)\tan(\theta),$$ 
then the first eigenvalue of problem \eqref{eigenvalue:problem} satisfies 
$\lambda_1(g,\theta)\geq n$. 
\end{theorem}
\begin{proof}
As in \cite[p.227]{XinHongwei}, setting $u=\frac{\d f}{\d \nu}\Big|_{\Sigma}$ and $z=f|_{\Sigma}$, and using \cite{Re1}, we obtain
\begin{align*}
\int_M((\Delta f)^2-|\nabla^2f|^2)dv_g
=&\int_{\Sigma}\big((\bar\Delta z+H_{\Sigma}u)u-<\bar\nabla z,\bar\nabla u>+h_{\Sigma}(\bar\nabla z,\bar\nabla z)\big)d\sigma
\\
&+\int_M\text{Ric}_g(\nabla f,\nabla f)dv_g
\end{align*}
where the bars stand for the restrictions to $\Sigma$. If $f$ satisfies
$\Delta_gf+\lambda_1(g,\theta)f=0$ and $u+z\cot(\theta)=0\:\: \text{on}\:\Sigma$, using $\text{Ric}_g\geq (n-1)g$ we obtain
\begin{align*}
\frac{n-1}{n}\lambda_1(g,\theta)(\lambda_1(g,\theta)-n)\int_Mf^2dv_g
\geq&\int_{\Sigma}(2\cot(\theta)|\bar\nabla z|^2+h_{\Sigma}(\bar\nabla z,\bar\nabla z))d\sigma
\\
&+\int_{\Sigma}(H_{\Sigma}-(n-1)\tan(\theta))u^2d\sigma.
\end{align*}
From the hypotheses on $h_{\Sigma}$ and $H_{\Sigma}$ we conclude that 
$\lambda_1(g,\theta)(\lambda_1(g,\theta)-n)\geq0$.
\end{proof}
As a consequence of Theorem \ref{main:thm} we prove:
\begin{theorem}
Equality  $\lambda_1(g,\theta)=n$  in Theorem \ref{thm:XinHongwei} holds if and only if $M$ is isometric to an $n$-dimensional spherical cap of $\mathbb{S}^n$, whose boundary is a totally umbilical hypersphere with mean curvature $(n-1)\tan(\theta)$.
\end{theorem}
\begin{proof}
Note from the previous proof that the equality holds if and only if the following four conditions hold:

\begin{itemize}
\item[(i)] $\text{Ric}_g(\nabla f, \nabla f)=(n-1)|\nabla f|^2$;
\item[(ii)] $H_{\Sigma}=(n-1)\tan(\theta)$;
\item[(iii)] $2\cot(\theta)|\bar\nabla z|^2+h_{\Sigma}(\bar\nabla z,\bar\nabla z)=0$;
\item[(iv)] $|\nabla^2f|^2=\frac{(\Delta f)^2}{n}$.
\end{itemize}
We can see that (iv) is equivalent to
\[
\nabla^2f=\frac{\Delta f}{n}g\,.
\]
Then we apply Theorem \ref{main:thm} to get the result.
\end{proof}




\end{document}